\documentclass[11pt,reqno]{amsart}
\usepackage[margin=1in]{geometry}

\usepackage{amsthm, amsmath, amssymb,booktabs,xcolor,graphicx}
\usepackage{tikz}
\usepackage{bbm}
\usepackage{listings}
\usepackage{mathtools}
\usepackage{xcolor}
\usepackage{appendix}
\usepackage{subfig}
\usepackage{mathscinet}
\usepackage{fourier}
\makeatletter
\def\resetMathstrut@{%
  \setbox\z@\hbox{%
    \mathchardef\@tempa\mathcode`\(\relax
    \def\@tempb##1"##2##3{\the\textfont"##3\char"}%
    \expandafter\@tempb\meaning\@tempa \relax
  }%
  \ht\Mathstrutbox@1.2\ht\z@ \dp\Mathstrutbox@1.2\dp\z@
}
\makeatother
\definecolor{codegreen}{rgb}{0,0.6,0}
\definecolor{codegray}{rgb}{0.5,0.5,0.5}
\definecolor{codepurple}{rgb}{0.58,0,0.82}
\definecolor{backcolour}{rgb}{0.95,0.95,0.92}

\lstdefinestyle{mystyle}{
  backgroundcolor=\color{backcolour},   commentstyle=\color{codegreen},
  keywordstyle=\color{magenta},
  numberstyle=\tiny\color{codegray},
  stringstyle=\color{codepurple},
  basicstyle=\ttfamily\footnotesize,
  breakatwhitespace=false,
  breaklines=true,
  captionpos=b,
  keepspaces=true,
  numbers=left,
  numbersep=5pt,
  showspaces=false,
  showstringspaces=false,
  showtabs=false,
  tabsize=2
}

\usepackage[colorlinks=true,allcolors=blue,backref=page]{hyperref}

\usepackage[utf8]{inputenc}
\usepackage[T1]{fontenc}

\usepackage[textsize=small,backgroundcolor=orange!20]{todonotes}

\usepackage{url}
\usepackage{dutchcal}

\usepackage[noabbrev,capitalize,nameinlink]{cleveref}
\crefformat{equation}{(#2#1#3)}
\crefrangeformat{equation}{(#3#1#4) to~(#5#2#6)}
\crefmultiformat{equation}{(#2#1#3)}
{ and~(#2#1#3)}{, (#2#1#3)}{ and~(#2#1#3)}

\usepackage[color,final]{showkeys} 

\colorlet{refkey}{orange!20}
\colorlet{labelkey}{blue!30}

\newtheorem{theorem}{Theorem}
\newtheorem{proposition}[theorem]{Proposition}
\newtheorem{lemma}[theorem]{Lemma}

\newtheorem{corollary}[theorem]{Corollary}

\newtheorem*{question*}{Question}

\theoremstyle{definition}
\newtheorem{definition}[theorem]{Definition}
\newtheorem{problem}[theorem]{Problem}

\newtheorem*{definition*}{Definition}

\theoremstyle{remark}

\newcommand{\abs}[1]{\left\lvert#1\right\rvert}

\newcommand{\norm}[1]{\left\lVert#1\right\rVert}

\newcommand{\sang}[1]{\langle #1 \rangle}

\newcommand{\CC}{\mathbb{C}}
\newcommand{\EE}{\mathbb{E}}

\newcommand{\cB}{\mathcal{B}}

\newcommand{\HS}{\mathrm{HS}}

\newcommand{\mb}{\mathbb}

\newcommand{\mr}{\mathrm}
\newcommand{\ol}{\overline}

\newcommand{\wh}{\widehat}

\renewcommand{\epsilon}{\varepsilon}
\newcommand{\e}{\varepsilon}
\renewcommand{\le}{\leqslant}
\renewcommand{\ge}{\geqslant}

\renewcommand{\setminus}{\smallsetminus}
\newcommand{\C}{\mathbb{C}}
\newcommand{\R}{\mathbb{R}}
\newcommand{\U}{\mathbb{U}}
\newcommand{\h}{\mathcal{h}}
\renewcommand{\subset}{\subseteq}
\newcommand{\1}{\mathbf 1}
\newcommand{\N}{\mathbb{N}}
\newcommand{\Z}{\mathbb{Z}}
\newcommand{\Q}{\mathbb{Q}}
\newcommand{\K}{\mathbb{K}}
\newcommand{\ud}[0]{\,\mathrm{d}}
\newcommand{\E}{\mathbb{E}}
\newcommand{\M}{\mathsf{M}}

\newcommand{\f}{\varphi}
\newcommand{\Tr}{\mathrm{Tr}}
\newcommand{\g}{\mathcal{g}}
\renewcommand{\i}{\mathsf{i}}
\usepackage{appendix}

\title{Cayley graphs that have a quantum ergodic eigenbasis}

\author[Naor]{Assaf Naor}
\author[Sah]{Ashwin Sah}
\author[Sawhney]{Mehtaab Sawhney}
\author[Zhao]{Yufei Zhao}

\thanks{Naor was supported by NSF grant DMS-2054875 and a Simons Investigator award.
Sah and Sawhney were supported by NSF Graduate Research Fellowship Program DGE-1745302.
Sah was supported by the PD Soros Fellowship.
Zhao was supported by NSF CAREER Award DMS-2044606, a Sloan Research Fellowship, and the MIT Solomon Buchsbaum Fund.}

\address{Department of Mathematics, Princeton NJ 08544-1000}
\email{naor@math.princeton.edu}

\address{Department of Mathematics, Massachusetts Institute of Technology, Cambridge, MA 02139, USA}
\email{\{asah,msawhney,yufeiz\}@mit.edu}

\begin{document}

\begin{abstract}
We investigate  which finite Cayley graphs admit a quantum ergodic eigenbasis, proving that this holds for any Cayley graph on a group of size $n$ for which the sum of the dimensions of its irreducible representations is $o(n)$, yet there exist Cayley graphs that do not have any quantum ergodic eigenbasis.
\end{abstract}

\maketitle

\section{Introduction}\label{sec:introduction}

We will prove here the following theorem; see Theorem~\ref{thm:not-quantum-ergodic} for a companion impossibility result.

\begin{theorem} \label{thm:qe-2}
There exists an absolute constant $c>0$ with the following property. Suppose that $\e>0$ and let  $G$ be a finite group whose irreducible representations have total dimension at most $c \epsilon^2 \abs{G}$, i.e.,
\begin{equation}\label{eq:total dimension condition}
\sum_{\sigma\in \wh G} d_\sigma\le c\e^2 \abs{G}.
\end{equation}
Then, any Cayley graph on $G$ has an orthonormal eigenbasis $\cB$ consisting of functions $\phi\colon G \to \CC$ satisfying
\begin{equation}\label{eq:2QE}
\forall f \colon G \to \CC,\qquad \EE_{\phi \in \cB} \Big[\big| \EE_{x \in G} \big[f(x) |\phi(x)|^2\big] - \mb{E} f\big|\Big] \le \epsilon\|f\|_2.
\end{equation}
\end{theorem}

In the statement  of Theorem~\ref{thm:qe-2}, $\wh G$ is the set of irreducible unitary representations of a finite group $G$ and the dimension of each $\sigma \in \wh G$ is denoted $d_\sigma$. The  representation theory of finite groups that we will use below is rudimentary; see e.g.~\cite{HR70,Sim96}. A Cayley graph on $G$ is a  graph whose vertex set is $G$ such that there is a symmetric subset $\mathfrak{S}\subset G$ that generates $G$ and $\{g,h\}\subset G$ forms an edge if and only if $gh^{-1}\in \mathfrak{S}$.

In  Theorem~\ref{thm:qe-2} and throughout what follows, we will adhere to the convention that a finite set $X$ is only equipped with the uniform probability measure; thus, all expectations, scalar products and $L_p$ norms of functions from $X$ to $\C$ will be with respect to this measure, i.e., for every $f,g\colon X\to \C$ and $1\le p\le \infty$,
\begin{equation}\label{eq:normalization}
\EE f = \EE_{x\in X}\big[f(x)\big]=\frac{1}{\abs{X}} \sum_{x\in X} f(x)\qquad\mathrm{and}\qquad \sang{f,g} = \EE_{x \in X} \Big[\ol{f(x)}g(x)\Big]\qquad \mathrm{and}\qquad \norm{f}_p = \Big(\EE \big[|f|^p\big]\Big)^{\frac{1}{p}}.
\end{equation}
So, a set of functions $\phi_1, \dots, \phi_{\abs{X}} \colon X \to \CC$ is an orthonormal basis if $\|\phi_j\|_2=1$ and $\langle \phi_j, \phi_k\rangle = 0$ for every distinct $j,k\in \{1,\ldots,|X|\}$. If $X$ is a graph, then  we say that $\cB=\{\phi_1, \dots, \phi_{\abs{X}}\}$ is an orthonormal eigenbasis of $X$ if it is an orthonormal basis consisting of eigenfunctions of the adjacency matrix of $X$.

Theorem~\ref{thm:qe-2} is a finitary statement in the spirit of quantum ergodicity on manifolds, e.g.~\v{S}nirel\cprime man's classical theorem~\cite{Sni74,CdV85,Zel87}. Investigations along these lines include notably~\cite{AL15}, and we refer also to~\cite{Ana17,AS17}  and the survey~\cite{Ana18} for background and motivation. From these works, we extract the following definition.

\begin{definition}[quantum ergodic basis]\label{def:qe}
Given a finite set $X$ and $\e>0$, we say that an orthonormal basis $\cB$ of functions $\phi\colon X \to \CC$
is \emph{$\epsilon$-quantum ergodic} if
\begin{equation}
\label{eq:qe}
\forall f \colon X \to \CC,\qquad \EE_{\phi \in \cB} \Big[\big| \EE_{x \in X} \big[f(x) |\phi(x)|^2\big] - \mb{E} f\big|\Big] \le \epsilon\|f\|_\infty.
\end{equation}
\end{definition}

The only difference between the conclusion~\eqref{eq:2QE} of Theorem~\ref{thm:qe-2}  and the requirement~\eqref{eq:qe} of Definition~\ref{def:qe} is that the quantity $\|f\|_2$ in the right hand side of~\eqref{eq:2QE}  is replaced in the right hand side of~\eqref{eq:qe} by the larger quantity $\|f\|_\infty$.  Therefore,  Theorem~\ref{thm:qe-2}  implies that  any Cayley graph of a finite group whose irreducible representations have total dimension at most $c \epsilon^2 \abs{G}$ has an $\epsilon$-quantum ergodic eigenbasis. The stronger conclusion~\eqref{eq:2QE} of Theorem~\ref{thm:qe-2} can be significantly stronger when e.g.~in~\eqref{eq:2QE} we take $f$ to be the indicator of a small nonempty subset $S$ of $G$, as in this case $\|f\|_\infty=1$  while $\|f\|_2=\sqrt{|S|/|G|}$.

The reason why we formulated Definition~\ref{def:qe}  using the $L_\infty$ norm of $f$ rather than its $L_2$ norm is first and foremost because this is how the subject is treated in the literature, but also because the following impossibility result rules out even the weaker requirement~\eqref{eq:qe}.

\begin{theorem}\label{thm:not-quantum-ergodic}
There are arbitrarily large Cayley graphs that do not admit any $c$-quantum ergodic orthonormal eigenbasis, where $c>0$ is a universal constant.
\end{theorem}

The groups that we will construct in the proof of Theorem~\ref{thm:not-quantum-ergodic} will be a direct product of a cyclic group with an appropriately chosen fixed group (specifically, a group that was constructed in~\cite{SSZ22}).


\begin{problem}\label{prob:main open} For a finite group $G$ let $\e(G)$ be the infimum over  $\e>0$ such that every Cayley graph on $G$ has an $\e$-quantum ergodic orthonormal eigenbasis. Characterize those sequences $\{G_n\}_{n=1}^\infty$ of  groups for which $\lim_{n\to \infty} \e(G_n)=0$. More ambitiously, how can one compute $\e(G)$ up to universal constant factors?
\end{problem}

Any Abelian group $G$ satisfies $\e(G)=0$,  as seen by considering the eigenbasis $\cB$ of Fourier characters:  each $\phi\in \cB$  takes value among the roots of unity, so the left-hand side of \cref{eq:2QE} vanishes for every $f:G\to \C$. Theorem~\ref{thm:qe-2} furnishes many more examples of sequences $\{G_n\}_{n=1}^\infty$ of groups with  $\lim_{n\to \infty} \e(G_n)=0$.

If $\eta>0$ and $G$ is a group with at most $\eta \abs{G}$ conjugacy classes (e.g.~by~\cite[Theorem~2]{LB97} this holds with $\eta=2^{n-1}/|G|$ if $G$ is any subgroup of the permutation group $S_n$), then every Cayley graph on $G$ has a $O(\sqrt[4]{\eta})$-quantum ergodic orthonormal eigenbasis. Indeed,
$$
\sum_{\sigma\in \wh G} d_\sigma\le |\wh{G}|^{\frac12}\bigg(\sum_{\sigma\in \wh G} d_\sigma^2\bigg)^{\frac12}\le \sqrt{\eta} |G|,
$$
where the first step uses Cauchy–-Schwarz and  that $|\wh{G}|$  equals the  number of conjugacy classes of $G$, and the second step uses the above assumption and that $\sum_{\sigma\in \wh G} d_\sigma^2=|G|$. Hence, \eqref{eq:total dimension condition} holds with $\e=\sqrt[4]{\eta}/\sqrt{c}$.

A special case of the above example is  when for some $D\in \N$ a group $G$ is nontrivial and $D$-quasirandom in the sense  of Gowers~\cite{Gow08}, i.e., every nontrivial unitary representation of $G$ has dimension at least $D$. This implies that $G$ has at most $2|G|/(D^2+1)$ conjugacy classes, and hence   every Cayley graph on $G$ has a $O(1/\sqrt{D})$-quantum ergodic orthonormal eigenbasis. Indeed,
\begin{equation}\label{eq:sum of squares quasirandom}
|G|= \sum_{\sigma \in \wh G} d_\sigma^2= 1 + \sum_{\sigma \in \wh G \setminus \{\text{triv}\}} d_\sigma^2\ge 1+D^2\big(|\wh{G}|-1\big).
\end{equation}
Thus,  $|\wh{G}|\le 1+(|G|-1)/D^2\le 2|G|/(D^2+1)$, where the last step holds as $|G|> 1$ and therefore the second sum in~\eqref{eq:sum of squares quasirandom} is nonempty, so in fact $|G|\ge D^2+1$. By an inspection of the tables on pages 769--770 of~\cite{Col08} and the classification of finite simple groups, if $G$ is a non-cyclic simple group, then we can take   $D$ to be at least a universal constant multiple of $(\log |G|)/\log\log |G|$; for most simple groups a much better lower bound on $D$ is available, and many more examples appear in the literature (see e.g.~\cite[Chapter~1, \S1.3]{Tao15}).

At the same time, Theorem~\ref{thm:not-quantum-ergodic} demonstrates that  some assumption on $\{G_n\}_{n=1}^\infty$ must be imposed to ensure that $\lim_{n\to \infty} \e(G_n)=0$. Thus, Problem~\ref{prob:main open} remains an intriguing open question.


\section{Proof of Theorem~\ref{thm:qe-2}}

The Haar probability measure on a compact topological group $\Gamma$  will be denoted $\h_\Gamma$. Given $d\in \N$, the standard coordinate basis of $\C^d$ will be denoted $e_1,\ldots,e_d$ and the unitary group of $d\times d$ matrices will be denoted $\U(d)$. The Hilbert--Schmidt norm of a $d\times d$ matrix $A=(a_{jk})\in \M_d(\C)$ will be denoted
$$
\|A\|_\HS=\bigg(\sum_{j=1}^d\sum_{k=1}^d |a_{jk}|^2\bigg)^{\frac12}.
$$

Our construction of the basis $\mathcal{B}$ of Theorem~\ref{thm:qe-2} will be randomized; its main probabilistic input is the following lemma whose proof appears in Section~\ref{sec:probabilisitc} below.

\begin{lemma}\label{lem:main probabilisitc tool}There exists a universal constant $0<\eta<1$ with the following property. Let $S$ be a finite set. For every $s\in S$ fix an integer $d_s\in \N$  and a $d_s\times d_s$ matrix $A_s\in \M_{d_s}(\C)$ whose trace satisfies $\Tr(A_s)=0$. Denote \begin{equation}\label{eq:T in lemma}
\alpha=\left(\frac{\sum_{s\in S} \frac{1}{d_s}\|A_s\|_\HS^2}{\sum_{s\in S}d_s}\right)^\frac12\qquad\mathrm{and} \qquad T=\bigcup_{s\in S} \big(\{s\}\times \{1,\ldots,d_s\}\big)=\big\{(s,k):\ s\in S\ \wedge k\in \{1,\ldots,d_s\}\big\}.
\end{equation} Consider the direct product $\Gamma=\prod_{s\in S}\U(d_s)$ of the unitary groups $\{\U(d_s)\}_{s\in S}$. Then, for every $\beta\ge 2$ we have
\begin{equation}\label{eq:general matrix inequality}
 \h_\Gamma\Big[\big\{U=(U_s)_{s\in S}\in \Gamma:\ \E_{(s,k)\in T}\big[|e_k^*U_s^*A_sU_se_k |\big]\ge \beta\alpha \big\}\Big]\le e^{-\eta \beta^2\sum_{s\in S} d_s}.
\end{equation}
\end{lemma}

  Fix a finite group $G$  and fix also a symmetric subset $\mathfrak{S}\subset G$  that generates $G$.  Let $n=|G|$. The adjacency matrix $A(G;\mathfrak{S})\in \M_n(\{0,1\})$ of the Cayley graph that is induced by $\mathfrak{S}$ on $G$ acts on a function $f\colon G\to \C$ by $A(G;\mathfrak{S})f(x)=\sum_{\sigma\in \mathfrak{S}} f(\sigma x)$ for every $x\in G$.

We will apply Lemma~\ref{lem:main probabilisitc tool} with the index set $S$
$$
S=\bigcup_{\rho\in \wh{G}}\big(\{\rho\}\times \{1,\ldots,d_\rho\}\big)=\big\{(\rho,j):\ \rho\in \wh{G}\ \wedge \ j\in \{1,\ldots,d_\rho\}\big\}.
$$
and $d_s=d_\rho$ for every $s=(\rho,j)\in S$.  For this $S$, the set $T$ in~\eqref{eq:T in lemma} becomes
$$
T=\big\{(\rho,j,k):\ \rho\in \wh{G}\ \wedge\ (j,k)\in \{1,\ldots,d_\rho\}^2\big\}.
$$
Henceforth, $\Gamma=\prod_{(\rho,j)\in S}\U(d_\rho)\cong \prod_{\rho\in \wh{G}} \U(d_\rho)^{d_{\rho}}$ will be the group from Lemma~\ref{lem:main probabilisitc tool}.

Suppose that for each $\rho\in \wh{G}$ and $j,k\in \{1,\ldots,d_\rho\}$ we are given  $a_{\rho,j,k}\in \C^{d_\rho}$ and $b_{\rho,j}\in \C^{d_\rho}$ such that
\begin{equation}\label{eq:non-normalized orthogonality of b}
\forall j,j',k,k'\in \{1,\ldots,d_\rho\},\qquad a_{\rho,j,k}^*a_{\rho,j,k'}^{\phantom{*}}=\1_{\{k=k'\}} \qquad\mathrm{and}\qquad  b_{\rho,j}^*b_{\rho,j'}^{\phantom{*}}=\1_{\{j=j'\}}.
\end{equation}
This is an orthornormality requirement\footnote{To be consistent with our normalization convention in~\eqref{eq:normalization}, for every $d\in \N$ we will use matrix notation as in~\eqref{eq:non-normalized orthogonality of b}  when treating the standard scalar product on $\C^d$.} with respect to the standard (not normalized) scalar product on $\C^{d_\rho}$. The statement of Schur orthogonality is that whenever~\eqref{eq:non-normalized orthogonality of b} holds the following collection of functions  from $G$ to $\C$ (indexed by $T$) is orthonormal; as $|T|=\sum_{\rho\in \wh{G}} d_\rho^2=n$, it is an orthonormal basis of $G$:
\begin{equation}\label{eq:ab version}
\Big\{(x\in G)\mapsto d_\rho^\frac12 a_{\rho,j,k}^*\rho(x)^*b_{\rho,j}\Big\}_{(\rho,j,k)\in T}.
\end{equation}
These expressions are also natural through the lens of non-Abelian Fourier analysis.
It is mechanical to check that~\eqref{eq:ab version} consists of eigenfunctions of the adjacency matrix $A(G;\mathfrak{S})$ if for each $\rho\in \wh{G}$ we choose $b_{\rho,1},\ldots,b_{\rho,d_\rho}\in \C^{d_\rho}$ to be eigenvectors of the (Hermitian, as $\mathfrak{S}$ is symmetric) matrix
$$
\wh{\1_{\mathfrak{S}}}(\rho)=\E_{\sigma\in \mathfrak{S}}\big[\rho(\sigma)\big]\in \M_{d_\rho}(\C).
$$
So, we will henceforth assume that  $\{b_{\rho,j}\}_{j=1}^{d_\rho}$ are eigenvectors of $\wh{\1_{\mathfrak{S}}}(\rho)$ and satisfy~\eqref{eq:non-normalized orthogonality of b} for each $\rho\in \wh{G}$. 

We will prove Theorem~\ref{thm:qe-2} by choosing the rest of the datum in~\eqref{eq:ab version} uniformly at random. Namely, vectors $\{a_{\rho,j,k}\}_{(\rho,j,k)\in T}$ as above can be parameterized by taking $U=(U_{\rho,j})_{(\rho,j)\in S}\in \Gamma$ and letting $a_{\rho,j,k}=U_{\rho,j}e_k$ for every $(\rho,j,k)\in T$. Using this notation, the orthonormal eigenbasis of $G$ in~\eqref{eq:ab version}  becomes 
\begin{equation*}\label{eq:def BU}
\mathcal{B}_U=\Big\{(x\in G)\mapsto d_\rho^\frac12 e_k^*U_{\rho,j}^*\rho(x)^*b_{\rho,j}\Big\}_{(\rho,j,k)\in T}.
\end{equation*}
We will show that if~\eqref{eq:total dimension condition} holds and $U\in\Gamma$ is distributed according to the Haar probability measure $\h_\Gamma$, then  $\mathcal{B}_U$ satisfies  the conclusion of Theorem~\ref{thm:qe-2} with probability at least $1-e^{-n}$.

We will see that the following lemma is an instantiation of Lemma~\ref{lem:main probabilisitc tool}.

\begin{lemma}\label{lem:fixed f} Let $\eta>0$ be the universal constant of Lemma~\ref{lem:main probabilisitc tool}. For every $\beta\ge 2$ and $f:G\to \C$ we have
\begin{equation}\label{eq:desired fixed f}
\h_\Gamma\bigg[\Big\{U\in \Gamma:\  \EE_{\phi \in \cB_U} \Big[\big| \EE_{x \in G} \big[f(x) |\phi(x)|^2\big] - \mb{E} f\big|\Big] \ge \beta\bigg(\frac{1}{n}\sum_{\rho\in \wh{G}} d_\rho\bigg)^\frac12\|f\|_2\Big\}\bigg]\le e^{-\eta \beta^2 n}.
\end{equation}
\end{lemma}

Prior to proving  Lemma~\ref{lem:fixed f}, we will explain how it implies Theorem~\ref{thm:qe-2}.

\begin{proof}[Deduction of Theorem~\ref{thm:qe-2} from Lemma~\ref{lem:fixed f}] It is a classical fact (see e.g.~\cite[Lemma~2.4]{FLM77}) that there exist $f_1,\ldots, f_{5^{2n}}:G\to \C$ with $\|f_1\|_2=\ldots=\|f_{5^{2n}}\|_2=1$ such that every $f:G\to \C$ with $\|f\|_2=1$ belongs to the convex hull of $\{2 f_1,\ldots,2f_{5^{2n}}\}$ (better bounds on such polytopal approximation of balls can be found in~\cite{BW03,Bar14,NNR20}, but they only affect the  constant $c$ in Theorem~\ref{thm:qe-2}). Since for every fixed $U\in \Gamma$ the mapping
$$
(f:G\to \C)\mapsto  \EE_{\phi \in \cB_U} \Big[\big| \EE_{x \in G} \big[f(x) |\phi(x)|^2\big] - \mb{E} f\big|\Big]
$$
is convex (in the variable $f$), it follows that
\begin{equation*}
\sup_{\substack{f:G\to \C\\ \|f\|_2= 1}} \EE_{\phi \in \cB_U} \Big[\big| \EE_{x \in G} \big[f(x) |\phi(x)|^2\big] - \mb{E} f\big|\Big]\le
2\max_{\ell\in \{1,\ldots,5^{2n}\}} \EE_{\phi \in \cB_U} \Big[\big| \EE_{x \in G} \big[f_\ell(x) |\phi(x)|^2\big] - \mb{E} f_\ell\big|\Big].
\end{equation*}
Consequently, if $\eta$ is the universal constant in~\eqref{eq:desired fixed f}, then
\begin{align*}
\h_\Gamma\bigg[&\Big\{U\in \Gamma:\ \forall f:G\to \C,\quad  \EE_{\phi \in \cB_U} \Big[\big| \EE_{x \in G} \big[f(x) |\phi(x)|^2\big] - \mb{E} f\big|\Big] \le \frac{5}{\sqrt{\eta}}\bigg(\frac{1}{n}\sum_{\rho\in \wh{G}} d_\rho\bigg)^\frac12\|f\|_2\Big\}\bigg]\\&\ge 1-\sum_{\ell=1}^{5^{2n}} \h_\Gamma\Bigg[\bigg\{U\in \Gamma:\ \forall \ell\in \{1,\ldots, 5^{2n}\},\quad  \EE_{\phi \in \cB_U} \Big[\big| \EE_{x \in G} \big[f_\ell(x) |\phi(x)|^2\big] - \mb{E} f_\ell\big|\Big] \ge \frac{5}{2\sqrt{\eta}}\bigg(\frac{1}{n}\sum_{\rho\in \wh{G}} d_\rho\bigg)^\frac12\bigg\}\Bigg]\\
&\ge 1-5^{2n}\cdot e^{-5n}\ge 1-e^{-n}>0.
\end{align*}
Hence, there is $U\in \Gamma$ such that if~\eqref{eq:total dimension condition} holds with $c=\sqrt{\eta}/5$, then the orthonormal eigenbasis $\cB_U$ satisfies
\begin{equation*}
\forall f \colon G \to \CC,\qquad \EE_{\phi \in \cB_U} \Big[\big| \EE_{x \in G} \big[f(x) |\phi(x)|^2\big] - \mb{E} f\big|\Big] \le \frac{5}{\sqrt{\eta}}\bigg(\frac{1}{n}\sum_{\rho\in \wh{G}} d_\rho\bigg)^\frac12\|f\|_2\le \frac{5c}{\sqrt{\eta}}\e\|f\|_2=\e\|f\|_2.\tag*{\qedhere}
\end{equation*}
\end{proof}

We will next prove Lemma~\ref{lem:fixed f} assuming Lemma~\ref{lem:main probabilisitc tool}, after which we will pass (in Section~\ref{sec:probabilisitc}) to the proof of Lemma~\ref{lem:main probabilisitc tool}, thus completing the proof of Theorem~\ref{thm:qe-2}.

\begin{proof}[Deduction of Lemma~\ref{lem:fixed f} from Lemma~\ref{lem:main probabilisitc tool}]  As $\|f-\E f\|_2\le \|f\|_2\le 1$, it suffices to prove~\eqref{eq:desired fixed f} under the additional assumptions $\E f=0$ and $\|f\|_2=1$. Observe that for every  $(\rho,j,k)\in T$ and $U\in \Gamma$ we have
$$
\E_{x\in G} \Big[f(x) \big|d_\rho^\frac12 e_k^*U_{\rho,j}^*\rho(x)^*b_{\rho,j}\big|^2\Big]= e_k^*U_{\rho,j}^* A_{\rho,j}^fU_{\rho,j}e_k,
$$
where we introduce the notation
$$
A_{\rho,j}^f=d_\rho\E_{x\in G} \big[f(x)\rho(x)^*b_{\rho,j}b_{\rho,j}^*\rho(x)\big]\in M_{d_\rho}(\C).
$$
For every  $(\rho,j)\in S$,
$$
\Tr \big[A_{\rho,j}^f\big]=d_\rho\E\Big[f(x)\Tr\big[\rho(x)^*b_{\rho,j}b_{\rho,j}^*\rho(x)\big]\Big]= d_\rho\big(\E f\big) \Tr\big[b_{\rho,j}b_{\rho,j}^*\big]=0,
$$
where we used the cyclicity  of the trace and that $\rho(x)$ is unitary for every $x\in G$.  Also,
\begin{equation}\label{eq:average over G2}
\big\|A_{\rho,j}^f\big\|_\HS^2=\Tr \big[\big(A_{\rho,j}^f\big)^*A^f_{\rho,j}\big]=d_\rho^2 \E_{(x,y)\in G\times G} \Big[\overline{f(x)} f(y) \Tr\big[\rho(x)^*b_{\rho,j}b_{\rho,j}^*\rho(x) \rho(y)^*b_{\rho,j}b_{\rho,j}^*\rho(y)\big]\Big].
\end{equation}
Using the cyclicity  of the trace once more, for every $x,y\in G$ we have
$$
\Tr\big[\rho(x)^*b_{\rho,j}b_{\rho,j}^*\rho(x) \rho(y)^*b_{\rho,j}b_{\rho,j}^*\rho(y)\big]=\big|b_{\rho,j}^*\rho(x)\rho(y)^*b_{\rho,j}\big|^2.
$$
In combination with~\eqref{eq:average over G2}, this gives that
\begin{multline*}
\big\|A_{\rho,j}^f\big\|_\HS^2=d_\rho^2 \E_{(x,y)\in G\times G} \bigg[\Big(\overline{f(x)b_{\rho,j}^*\rho(x)\rho(y)^*b_{\rho,j}}\Big)\Big(f(y) b_{\rho,j}^*\rho(x)\rho(y)^*b_{\rho,j} \Big)\bigg]\\\le  d_\rho^2\E_{(x,y)\in G\times G} \Big[|f(x)|^2 \big|b_{\rho,j}^*\rho(x)\rho(y)^*b_{\rho,j}\big|^2\Big]=  d_\rho\E_{x\in G} \bigg[|f(x)|^2 d_\rho\E_{y\in G} \Big[\big|b_{\rho,j}^*\rho(x)\rho(y)^*b_{\rho,j}\big|^2\Big]\bigg],
\end{multline*}
where the penultimate step uses Cauchy–-Schwarz. By Schur orthogonality,  for every $x\in G$ we have
$$
d_\rho\E_{y\in G} \Big[\big|b_{\rho,j}^*\rho(x)\rho(y)^*b_{\rho,j}\big|^2\Big]=\big((\rho(x)^*b_{\rho,j})^*\rho(x)^*b_{\rho,j}\big) \big(b_{\rho,j}^*b_{\rho,j}\big) =(b_{\rho,j}^*b_{\rho,j})^2=1.
$$
Therefore, $\big\|A_{s}^f\big\|_\HS^2\le d_\rho\|f\|_2\le  d_\rho$ for every $s\in S$. The desired estimate~\eqref{eq:desired fixed f} now follows from~\eqref{eq:general matrix inequality} because
\begin{equation*}
\sum_{s\in S} d_s=\sum_{\rho\in \wh{G}} d_\rho^2=n\qquad\mathrm{and}\qquad  \sum_{s\in S} \frac{1}{d_s}\big\|A_s^f\big\|_\HS^2\le |S|=\sum_{\rho\in \wh{G}} d_\rho \tag*{\qedhere}.
\end{equation*}
\end{proof}

\subsection{Concentration}\label{sec:probabilisitc} Given $d\in \N$, let $\g_d$ be the standard Riemannian metric on $\U(d)$, namely the geodesic distance that is induced by taking the Hilbert--Schmidt metric on all of the tangent spaces.

The following theorem is a concatenation of known results that we formulate for ease of later reference. Its quick justification below uses fundamental properties of logarithmic Soboloev inequalities~\cite{Gro75} on metric probability spaces; good expositions of what we need can be found in the monographs~\cite{Led01,Mec19}.

\begin{theorem}[concentration of measure on Pythagorean products of rescaled unitary groups]\label{thm:product} Let $S$ be a finite set and $\{d_s\}_{s\in S}\subset \N$.  Denote  $\Gamma=\mathbb{U}(d_1)\times \ldots\times \mathbb{U}(d_m)$.  Suppose that $K>0$ and that $f:\Omega\to \R$ satisfies
\begin{equation}\label{eq:pythgorean def}
\forall U=(U_s)_{s\in S}, V=(V_s)_{s\in S}\in \Gamma,\qquad |f(U)-f(V)|\le K \bigg(\sum_{s\in S} d_s \mathcal{g}_{d_s}(U_s,V_s)^2\bigg)^{\frac12}.
\end{equation}
In other words, \eqref{eq:pythgorean def} is the requirement that $f$ is $K$-Lipschitz with respect on the Pythagorean product of the metric spaces $\{(\mathbb{U}(d_s),\sqrt{d_s}\mathcal{g}_{d_1})\}_{s\in S}$. Then, for every $\e>0$ we have
\begin{equation}\label{eq:pythagorean concentration}
\h_\Gamma \left[f\ge \e+\int_\Gamma f\ud \h_\Gamma\right]\le \exp\left(-\frac{\e^2}{3\pi^2K^2}\right).
\end{equation}
\end{theorem}

\begin{proof} By the paragraph after Theorem~15 in~\cite{MM13}, for every $d\in \N$ the logarithmic Sobolev constant of the metric probability space $(\mathbb{U}(d),\mathcal{g}_{d},\mathcal{h}_{\U(d)})$ is at most $3\pi^2/(2d)$. As the logarithmic Sobolev constant scales quadratically with rescaling of the metric, it follows that the metric probability space $(\mathbb{U}(d),\sqrt{d}\mathcal{g}_{d},\mathcal{h}_{\U(d)})$ has logarithmic Sobolev  constant at most $3\pi^2/2$. By the tensorization property of the logarithmic Sobolev constant under Pythagorean products (see~\cite[Corollary~5.7]{Led01}), if we define
$$
\forall U=(U_s)_{s\in S}, V=(V_s)_{s\in S}\in \Omega,\qquad \rho(U,V)=\bigg(\sum_{s\in S} d_s \mathcal{g}_{d_s}(U_s,V_s)^2\bigg)^{\frac12},
$$
then the logarithmic Sobolev constant of the metric probability space $(\Gamma,\rho,\h_\Gamma)$ is at most $3\pi^2/2$. The desired  conclusion~\eqref{eq:pythagorean concentration} follows by the classical Herbst argument~\cite{DS84,AMS94,Led95} (see~\cite[Theorem~5.3]{Led01}).
\end{proof}

It is worthwhile to formulate separately the following quick corollary of Theorem~\ref{thm:product}.

\begin{corollary}\label{cor:sum of Lip} Continuing with the notation of Theorem~\ref{thm:product}, suppose that $\{K_s\}_{s\in S}\subset (0,\infty)$ and that for each $s\in S$ we are given a function $f_s:\U(d_s)\to \R$ that is $K_s$-Lipschitz with respect to the geodesic metric $\g_{d_s}$, i.e., $|f_s(U)-f_s(V)|\le K_s\g_{d_s}(U,V)$ for every $U,V\in \U(d_s)$. Then, for every $\e>0$ we have
\begin{equation*}
h_\Gamma \Bigg[\bigg\{U=(U_s)_{s\in S}\in \Gamma:\ \E_{s\in S}  \big[f_s(U_s)\big]\ge \E_{s\in S} \bigg[\int_{\U(d_s)} f_s\ud \h_{\U(d_s)}\bigg]+\e\bigg\}\Bigg]\le \exp\left(-\frac{\e^2|S|^2}{3\pi^2\sum_{s\in S} \frac{1}{d_s}K_s^2}\right).
\end{equation*}
\end{corollary}

\begin{proof} Define $f:\Gamma\to \R$ by setting $f(U)=\E_{s\in S}  \big[f_s(U_s)\big]$ for  $U=(U_s)_{s\in S}\in \Gamma$.  If $U, V\in \Gamma$, then
$$
|f(U)-f(V)|\le \E_{s\in S} \big[|f_s(U_s)-f_s(V_s)|\big]\le  \frac{1}{|S|}\sum_{s\in S} K_s g_s(U_s,V_s)\le \frac{1}{|S|}\bigg(\sum_{s\in S} \frac{1}{d_s}K_s^2\bigg)^\frac12\bigg(\sum_{s\in S} d_s \mathcal{g}_{d_s}(U_s,V_s)^2\bigg)^{\frac12},
$$
where the final step is  Cauchy–Schwarz. Now apply Theorem~\ref{thm:product}.
\end{proof}

The following lemma connects the above general discussion to Lemma~\ref{lem:main probabilisitc tool}.

\begin{lemma}\label{lem:sum of rotated matrix coeff} Suppose that $\f_1,\ldots,\f_d:\C\to \C$ are $1$-Lipschitz  and $A\in \M_n(\C)$. Define  $f:\U(d)\to \C$ by setting $$
\forall U\in \U(d),\qquad f(U)=\sum_{k=1}^d \f_k(e_k^*U^*AUe_k).
$$
Then, the Lipschitz constant of $f$ with respect to the geodesic distance $\g_d$ is at most $2\|A\|_{\HS}$, i.e.,
$$
\forall U,V\in \U(d),\qquad |f(U)-f(V)|\le 2\|A\|_\HS\g(U,V).
$$
\end{lemma}

\begin{proof} Fix $U,V\in \U(d)$. By the definition of $\g=\g_d(U,V)$, there is a smooth curve (unit-speed geodesic)  $\gamma:[0,\g]\to \U(d)$ that satisfies $\gamma(0)=U$, $\gamma(\g)=V$, and such that $\|\gamma'(t)\|_\HS=1$ for every $t\in [0,\g]$. Then,
\begin{multline*}
|f(U)-f(V)|\le \sum_{k=1}^d \big|\f_k(e_k^*U^*AUe_k)-\f_k(e_k^*V^*AVe_k)\big|\le  \sum_{k=1}^d \big|e_k^*U^*AUe_k- e_k^*V^*AVe_k\big|\\=
\sum_{k=1}^d \Big|\int_0^\g \frac{\ud}{\ud t}  \big(e_k^*\gamma(t)^*A\gamma(t)e_k \big)\ud t\Big| \le \int_0^\g \bigg(\sum_{k=1}^d  \Big|\frac{\ud}{\ud t}  \big(e_k^*\gamma(t)^*A\gamma(t)e_k \big)\Big|\bigg)\ud t.\\
\end{multline*}
It therefore  suffices to prove the following point-wise estimate:
\begin{equation}\label{eq:desired pointwise}
\forall t\in [0,\g],\qquad \sum_{k=1}^d\Big|\frac{\ud}{\ud t}  \big(e_k^*\gamma(t)^*A\gamma(t)e_k \big)\Big|\le 2\|A\|_\HS.
\end{equation}
This indeed holds because by Cauchy--Schwarz for every $t\in [0,\g]$ and $k\in \{1,\ldots,d\}$,
\begin{multline*}
\Big|\frac{\ud}{\ud t}  \big(e_k^*\gamma(t)^*A\gamma(t)e_k \big)\Big|= e_k^*\gamma'(t)^*A\gamma(t)e_k  +e_k^*\gamma(t)^*A\gamma'(t)e_k\\ \le
\big(e_k^*\gamma'(t)^*\gamma'(t)e_k\big)^\frac12 \big(e_k^*\gamma(t)^*A^*A\gamma(t)e_k\big)^\frac12 + \big(e_k^*\gamma(t)^*AA^*\gamma(t)e_k\big)^\frac12 \big(e_k^*\gamma'(t)^*\gamma'(t)e_k\big)^\frac12.
\end{multline*}
By summing this over $k\in \{1,\ldots,d\}$ and using Cauchy--Schwarz, we conclude the proof of~\eqref{eq:desired pointwise} as follows.
\begin{align*}
\sum_{k=1}^d\Big|\frac{\ud}{\ud t}  \big(e_k^*\gamma(t)^*A\gamma(t)e_k \big)\Big| &\le \Big(\sum_{k=1}^d e_k^*\gamma'(t)^*\gamma'(t)e_k\Big)^\frac12\bigg( \Big(\sum_{k=1}^d e_k^*\gamma(t)^*A^*A\gamma(t)e_k\Big)^\frac12+\Big(\sum_{k=1}^d e_k^*\gamma(t)^*AA^*\gamma(t)e_k\Big)^\frac12\bigg)\\&= \Big(\Tr\big[\gamma'(t)^*\gamma'(t)\big]\Big)^\frac12\bigg(\Big(\Tr\big[\gamma(t)^*A^*A\gamma(t)\big]\Big)^\frac12
+\Big(\Tr\big[\gamma(t)^*AA^*\gamma(t)\big]\Big)^\frac12\bigg)\\&= \Big(\Tr\big[\gamma'(t)^*\gamma'(t)\big]\Big)^\frac12\bigg(\Big(\Tr\big[A^*A\big]\Big)^\frac12
+\Big(\Tr\big[AA^*\big]\Big)^\frac12\bigg)=2\|A\|_{\HS}.\tag*{\qedhere}
\end{align*}
\end{proof}

We can now prove  Lemma~\ref{lem:main probabilisitc tool}, thus completing the proof of Theorem~\ref{thm:qe-2}.

\begin{proof}[Proof of Lemma~\ref{lem:main probabilisitc tool}] For every $d\in \N$ and $k\in \{1,\ldots,d\}$ we have
\begin{equation}\label{eq:second moment A}
\forall A\in \M_d(\C),\qquad \int_{\U(d)}\big| e_k^*U^*AUe_k\big|^2\ud  \h_{\U(d)}(U)=\frac{\|A\|_\HS^2+|\Tr(A)|^2}{d(d+1)}.
\end{equation}
One checks~\eqref{eq:second moment A} by noting that  if $U$ is distributed according to the Haar measure on $\U(d)$, then $Ue_k$ is distributed according to the normalized surface area measure on $\{z\in \C^d:\ |z_1|^2+\ldots+|z_d|^2=1\}$, expanding the squares and substituting the resulting standard spherical integrals that are computed in e.g.~\cite{Fol01}.

Returning to the setting and notation of Lemma~\ref{lem:main probabilisitc tool}, for every $s\in S$ and $U\in \U(d_s)$ define
$$
f_s(U)=\sum_{k=1}^{d_s} \big|e_k^*U^*A_sUe_k\big|.
$$
By Lemma~\ref{lem:sum of rotated matrix coeff}, the assumption of Corollary~\ref{cor:sum of Lip}  holds with $K_s=2\|A_s\|_\HS$. By Cauchy--Schwarz and~\eqref{eq:second moment A},
$$
\int_{\U(d_s)} f_s\ud \h_{\U(d_s)}=\sum_{k=1}^{d_s} \int_{\U(d_s)} \big|e_k^*U^*A_sUe_k\big|\ud \h_{\U(d_s)}(U)\le \sum_{k=1}^{d_s} \bigg(\int_{\U(d_s)} \big|e_k^*U^*A_sUe_k\big|^2\ud \h_{\U(d_s)}(U)\bigg)^\frac12\le \|A_s\|_\HS.
$$
Using Cauchy--Schwarz and recalling the definition of $\alpha$ in~\eqref{eq:T in lemma}, we therefore have  
$$
\E_{s\in S} \bigg[\int_{\U(d_s)} f_s\ud \h_{\U(d_s)}\bigg]\le \E_{s\in S} \big[\|A_s\|_\HS\big]\le \frac{1}{|S|}\bigg(\sum_{s\in S} d_s\bigg)^{\frac12} \bigg(\sum_{s\in S} \frac{1}{d_s}\|A_s\|_\HS^2\bigg)^\frac12=\frac{\sum_{s\in S}d_s}{|S|}\alpha.
$$
Corollary~\ref{cor:sum of Lip}  therefore implies the following estimate for every $\beta\ge 2$:
$$
h_\Gamma \bigg[\Big\{U=(U_s)_{s\in S}\in \Gamma:\ \E_{s\in S}  \big[f_s(U_s)\big]\ge \frac{\sum_{s\in S} d_s}{|S|} \beta \alpha \Big\}\bigg]\le \exp\bigg(-\frac{(\beta-1)^2}{3\pi^2}\sum_{s\in S}d_s\bigg)\le \exp\bigg(-\frac{\beta^2}{12\pi^2}\sum_{s\in S}d_s\bigg). 
$$
This coincides with the desired estimate~\eqref{eq:general matrix inequality} with $\eta=1/(12\pi^2)$.  
\end{proof}

\section{Proof of Theorem~\ref{thm:not-quantum-ergodic}}

For the statement of the following proposition, observe that if $H$ is a finite group and $\mathfrak{S}$ a symmetric generating subset of $H$, then $\mathfrak{S}\times \{-1,1\}$ generates $H\times (\Z/m\Z)$ for any odd integer $m\in 1+2\N$. Indeed, if $(h,k)\in H\times (\Z/m\Z)$, then take $a\in \N$ and  $\sigma_1,\ldots,\sigma_a\in \mathfrak{S}$ such that $h=\sigma_1\cdots\sigma_a$. Since $m$ is odd, there exists $b\in \N$ such that $a+2b\equiv k\mod m$. We then have $(h,k)=(\sigma_1,1)\cdots(\sigma_a,1)(\sigma_1,1)^b(\sigma_1^{-1},1)^b$.

\begin{proposition}[from quantum ergodicity to existence of delocalized eigenfunctions]\label{thm:from QE to delocalization} Let $H$ be a finite group and fix a symmetric generating subset $\mathfrak{S}$  of $H$. There is $\ell=\ell(H,\mathfrak{S})\in \N$ with the following property. Let $p>3$ be a prime that does not divide $\ell$. Consider the direct product $G=H\times (\Z/p\Z)$. Suppose that the Cayley graph that is induced on $G$ by the generating set $\mathfrak{S}\times \{-1,1\}$  has an $\e$-quantum ergodic eigenbasis for some $\e>0$. Then, for every nonzero eigenvalue $\lambda$ of the Cayley graph that is induced on $H$ by $\mathfrak{S}$ there exists an eigenfunction $\psi:H\to \C$ whose eigenvalue is $\lambda$ and $0<\|\psi\|_\infty\le \sqrt{2(1+ 2|H|^3\e)}\|\psi\|_2$.
\end{proposition}

Prior to proving Proposition~\ref{thm:from QE to delocalization}, we will explain  how it implies  Theorem~\ref{thm:not-quantum-ergodic}. This deduction uses (a very small part of) the following theorem from~\cite{SSZ22}:

\begin{theorem}\label{thm:quote SSZ}There exists a universal constant $\kappa>0$ with the following property. For arbitrarily large $n\in \N$ there exists a group $H$ with $|H|=n$ and a symmetric generating subset $\mathfrak{S}$ of $H$ such that the adjacency matrix $A(H;\mathfrak{S})$ has a nonzero eigenvalue $\lambda$ with the property that $\|\psi\|_\infty/\|\psi\|_2 \ge \kappa\sqrt{\log n}/\log\log n$ for every nonzero eigenfunction $\psi$ of $A(H;\mathfrak{S})$ whose eigenvalue is $\lambda$.
\end{theorem}

The statement of Theorem~1.2 in~\cite{SSZ22}  coincides with Theorem~\ref{thm:quote SSZ}, except that it does not include the  assertion that the eigenvalue is nonzero, but this is stated in the proof of~\cite[Theorem~1.2]{SSZ22}.

\begin{proof}[Deduction of Theorem~\ref{thm:not-quantum-ergodic} from Proposition~\ref{thm:from QE to delocalization}] If Theorem~\ref{thm:not-quantum-ergodic} does not hold, then by Proposition~\ref{thm:from QE to delocalization} for any nonzero eigenvalue $\lambda$ of any finite  Cayley graph there is an eigenfunction $\psi$ of that Cayley graph whose eigenvalue is $\lambda$ and $\|\psi\|_\infty\le \sqrt{2}\|\psi\|_2$. This contradicts Theorem~\ref{thm:quote SSZ}.
\end{proof}

Our proof of Proposition~\ref{thm:from QE to delocalization} uses the following  basic lemma about algebraic numbers; the rudimentary facts from Galois theory and cyclotomic fields that appear in its proof can be found in e.g.~\cite{Lan90}.

\begin{lemma}\label{lem:cyclotomic}
Let $\K$ be a finite degree number field. There exists $\ell=\ell(\K)\in \N$ such that if $p>3$ is a prime that does not divide $\ell$, then $\cos(2\pi j/p)/\cos(2\pi k/p)\notin \K$ for all distinct $j, k\in\{0,1,\ldots,(p-1)/2\}$.
\end{lemma}

\begin{proof}
Denote $\mb{Q}^{\mr{cyc}} = \mb{Q}(\{\exp(2\pi \i/k)\}_{k=1}^\infty)$. Let $\K' = \K\cap\mb{Q}^{\mr{cyc}}\subseteq \K$.
By the primitive element theorem, there exists $\alpha \in \K'$ such that $\K' = \mb{Q}(\alpha)$.
Since $\alpha\in\mb{Q}^{\mr{cyc}}$, there exists $\ell\in \N$ such that $\alpha\in\mb{Q}(\exp(2\pi \i/\ell))$. Therefore,
$\K\cap\mb{Q}^{\mr{cyc}}\subseteq\mb{Q}(\exp(2\pi \i/\ell))$. Observe that $\mb{Q}(\exp(2\pi \i/\ell))\cap\mb{Q}(\exp(2\pi \i/p)) = \mb{Q}$ for any prime $p$ that does not divide $\ell$ (as  the field  generated by $\mb{Q}(\exp(2\pi \i/\ell))$ and $\mb{Q}(\exp(2\pi \i/p))$  is $\mb{Q}(\exp(2\pi \i/(\ell p)))$, and its degree is $\varphi(\ell p) = \varphi(\ell)\varphi(p)$, where $\f(\cdot)$ is Euler's totient function, while  the degrees of $\mb{Q}(\exp(2\pi \i/\ell))$ and $\mb{Q}(\exp(2\pi \i/p))$ are, respectively, $\f(\ell)$ and $\f(p)$). Therefore
\begin{equation}\label{eq:field intersection}
\K\cap\mb{Q}\Big(e^{\frac{2\pi \i}{p}}\Big)
= \big(\K\cap\mb{Q}^{\mr{cyc}}\big)\cap\mb{Q}\Big(e^{\frac{2\pi \i}{p}}\Big) 
 \subseteq\mb{Q}\Big(e^{\frac{2\pi \i}{\ell}}\Big)  \cap \mb{Q}\Big(e^{\frac{2\pi \i}{p}}\Big) 
=\mb{Q}.
\end{equation}

Denoting $\zeta = \exp(2\pi \i/p)$, it follows from~\eqref{eq:field intersection} that if $\cos(2\pi j/p)/\cos(2\pi k/p)=(\zeta^j+\zeta^{-j})/(\zeta^k+\zeta^{-k})\in \K$ for some distinct $j, k\in\{0,1,\ldots,(p-1)/2\}$, then actually in $(\zeta^j+\zeta^{-j})/(\zeta^k+\zeta^{-k})\in \Q$. This cannot happen for the following reason. Suppose that there are $a,b\in\mb{Z}\setminus\{0\}$ for which $a(\zeta^j+\zeta^{-j})-b(\zeta^k+\zeta^{-k}) = 0$.   Given $r\in(\mb{Z}/p\mb{Z})\setminus \{0\}$, we can apply the automorphism of $\mb{Q}(\zeta)$ which maps $\zeta$ to $\zeta^r$. Since $p > 3$, we can choose $r$ so that $jr,kr\not\equiv(p-1)/2\pmod{m}$. We therefore deduce that $a(\zeta^{u}+\zeta^{-u})-b(\zeta^{v}+\zeta^{-v}) = 0$
for some distinct integers $0\le u,v < (p-1)/2$. Without loss of generality, $u < v$. Then $a(\zeta^{u+v}+\zeta^{v-u})-b(\zeta^{2v}+1)=0$. We have thus found  a nonzero  polynomial with integer coefficients of degree $2v < m-1$ that vanishes at $\zeta$, contradicting the fact that the minimal polynomial of $\zeta$ is $P(t) = t^{p-1}+\cdots+t+1$.
\end{proof}

We can now prove Proposition~\ref{thm:from QE to delocalization}, thus completing the proof of  Theorem~\ref{thm:not-quantum-ergodic}.

\begin{proof}[Proof of Proposition~\ref{thm:from QE to delocalization}]   Denote the distinct nonzero eigenvalues of the adjacency matrix $A(H;\mathfrak{S})$ by $\lambda_1,\ldots,\lambda_s\in \R\setminus\{0\}$, and for each $j\in \{1,\ldots,s\}$ let $\Lambda_j\subset \C^H$ be the eigenspace of  $A(H;\mathfrak{S})$ that corresponds to the eigenvalue $\lambda_j$. Also, let $\Lambda_0\subset \C^H$ be the kernel of $A(H;\mathfrak{S})$. Define
\begin{equation*}\label{eq:def M}
M=\max \bigg\{\inf_{\psi\in \Lambda_1\setminus \{0\}} \frac{\|\psi\|_\infty}{\|\psi\|_2},\ldots,\inf_{\psi\in \Lambda_s\setminus \{0\}} \frac{\|\psi\|_\infty}{\|\psi\|_2}\bigg\}.
\end{equation*}
The desired conclusion of Proposition~\ref{thm:from QE to delocalization} is the same as requiring that $M\le \sqrt{2(1+ 2|H|^3\e)}$. If $M\le \sqrt{2}$, then there is nothing to prove, so suppose from now on that $M>\sqrt{2}$.

Let $\ell$ be as in Lemma~\ref{lem:cyclotomic} applied to the field $\K=\Q(\lambda_1,\ldots,\lambda_s)$. Fix a prime $p>3$ that does not divide $\ell$ and let $G=H\times (\Z/p\Z)$ be as in the statement of Proposition~\ref{thm:from QE to delocalization}. For $k\in \Z$ denote $\mu_k=2\cos(2\pi k/p)$. As $p$ is odd, $\mu_k\neq 0$. Write $\chi_k(x)=\exp(2\pi\i kx/p)$ for  $x\in \Z/p\Z$ and let $E_k$ be the span of $\chi_k$ and $\chi_{-k}$ in $\C^{\Z/p\Z}$. Then, $\dim(E_0)=1$ and $\dim(E_k)=2$ for $k\in \{1,\ldots,(p-1)/2\}$, and $E_k$ is the eigenspace of $A(\Z/p\Z;\{-1,1\})$ whose eigenvalue is $\mu_k$. As $p$ is odd, the eigenspace decomposition of $A(\Z/p\Z;\{-1,1\})$ is
$$
\C^{\Z/p\Z}=\bigoplus_{k=0}^{\frac{p-1}{2}} E_k.
$$

The nonzero eigenvalues of $A(G,\mathfrak{S}\times \{-1,1\})$ are $\{\lambda_j\mu_k:\ (j,k)\in \{1,\ldots,s\}\times \{0,\ldots,(p-1)/2\}\}$; we claim that these numbers are distinct, so that the eigenspace decomposition of $A(G,\mathfrak{S}\times \{-1,1\})$ is
$$
\C^G\cong \C^H\otimes \C^{\Z/p\Z}= \Big(\Lambda_0\otimes \C^{\Z/p\Z}\Big)\bigoplus \Big( \bigoplus_{j=1}^s \bigoplus_{k=0}^{\frac{p-1}{2}} \Lambda_j\otimes E_k\Big).
$$
Indeed, if $j,j'\in \{1,\ldots,s\}$ and $k,k'\in \{1,\ldots, (p-1)/2\}$ are such that $\lambda_j\mu_k=\lambda_{j'}\mu_{k'}$, then $\mu_k/\mu_{k'}=\lambda_{j'}/\lambda_j\in \K$, so $k=k'$ by Lemma~\ref{lem:cyclotomic} and therefore also $j=j'$.  

Fix  $j\in \{1,\ldots,s\}$ at which  $M$ is attained, namely $\|\psi\|_\infty\ge M\|\psi\|_2$ for every $\psi\in \Lambda_j$. Let $\phi:G\to \C$ be an eigenfunction of $A(G,\mathfrak{S}\times \{-1,1\})$ whose eigenvalue is  $\lambda_j\mu_k$ for some $k\in \{0,\ldots, (p-1)/2\}$. So, $\phi \in \Lambda_j\otimes E_k$ and therefore there exist  $\psi_+,\psi_{-}\in \Lambda_j$ with $\|\psi_+\|_2^2+\|\psi_-\|_2^2=\|\phi\|_2^2$ such that $\phi=\psi_+\otimes \chi_k+\psi_{-}\otimes\chi_{-k}$. There is $\psi\in \{\psi_+,\psi_-\}$ with $\|\psi\|_2^2\ge \|\phi\|_2^2/2$. Fix  $h_\phi\in H$ for which $|\psi(h_\phi)|=\|\psi\|_\infty$. Then,
\begin{multline*}
\E_{x\in \Z/p\Z}\big[|\phi(h_\phi,x)|^2\big]=\E_{x\in \Z/p\Z} \Big[\big|\psi_+(h_\phi)e^{\frac{2\pi\i k x}{p}}+\psi_-(h_\phi)e^{-\frac{2\pi\i k x}{p}}\big|^2\Big]\\=|\psi_+(h_\phi)|^2+|\psi_-(h_\phi)|^2\ge |\psi(h_\phi)|^2=\|\psi\|_\infty^2\ge M^2\|\psi\|_2^2\ge \frac{M^2}{2}\|\phi\|_2^2.
\end{multline*}

If $\mathcal{B}\subset \C^G$ is an orthonormal eigenbasis of $A(G,\mathfrak{S}\times \{-1,1\})$, then let $\mathcal{B}'\subset \mathcal{B}$ be the elements of $\mathcal{B}$ whose  eigenvalue is $\lambda_j\mu_k$ for some $k\in \{0,\ldots, (p-1)/2\}$. Thus, $|\mathcal{B}'|=\dim(\Lambda_j)p\ge p$. By the pigeonhole principle there are $\mathcal{B}''\subset \mathcal{B}'$ and $h\in H$ such that $|\mathcal{B}''|\ge |\mathcal{B}'|/|H|\ge p/|H|$ and  $h_\phi=h$ for every $\phi\in \mathcal{B}''$. Consequently,
\begin{align}\label{eq:drop terms B''}
\begin{split}
\EE_{\phi \in \cB} \Big[\big| \EE_{x \in G} \big[&\1_{\{h\}\times \Z/p\Z}(x) |\phi(x)|^2\big] - \mb{E} \1_{\{h\}\times \Z/p\Z}\big|\Big]=\frac{1}{p|H|^2} \sum_{\phi\in \mathcal{B}} \Big|\E_{x\in \Z/p\Z}\big[|\phi(h_\phi,x)|^2\big]-1\Big|\\ &
\ge \frac{1}{p|H|^2}  \sum_{\phi\in \mathcal{B}''} \Big|\E_{x\in \Z/p\Z}\big[|\phi(h_\phi,x)|^2\big]-1\Big|\ge \frac{|\mathcal{B}''|}{p|H|^2}\left(\frac{M^2}{2}-1\right)\ge \frac{1}{|H|^3}\left(\frac{M^2}{2}-1\right).
\end{split}
\end{align}
If $\mathcal{B}$ is $\e$-quantum ergodic, then the first term in~\eqref{eq:drop terms B''} is at most $\e$, and therefore $M\le \sqrt{2(1+|H|^3\e)}$.
\end{proof}

\bibliographystyle{amsplain0.bst}
\bibliography{QE.bib}

\end{document}